\documentclass{amsart}

\usepackage{amsmath,amsthm,amssymb,amsfonts,enumerate,color}
\usepackage{mathrsfs}
\usepackage{amstext,amsxtra}
\usepackage[margin=1.5in]{geometry}
\usepackage{graphicx}
\usepackage{float}

\usepackage{tikz}
\usetikzlibrary{matrix,arrows,calc,intersections,fit}
\usetikzlibrary{decorations.markings}
\usepackage{tikz-cd}
\usepackage{textcomp}

\usepackage[colorlinks,urlcolor=black,linkcolor=blue,citecolor=blue,hypertexnames=false]{hyperref}
\allowdisplaybreaks
\usepackage{pgf,tikz}
\usepackage{stmaryrd}
\usepackage{bm}

{\theoremstyle{plain}
\newtheorem{theorem}{Theorem}[section]

\newtheorem{corollary}[theorem]{Corollary}
\newtheorem{lemma}[theorem]{Lemma}}
{\theoremstyle{definition}

\newtheorem{definition}[theorem]{Definition}
}
{\theoremstyle{remark}
\newtheorem{remark}[theorem]{Remark}}

\newcommand{\aut}{\operatorname{Aut}}
\newcommand{\conj}{\operatorname{conj}}

\newcommand{\sign}{\operatorname{sign}}
\newcommand{\sou}{\operatorname{source}}
\newcommand{\targ}{\operatorname{target}}
\newcommand{\even}{\operatorname{even}}
\newcommand{\odd}{\operatorname{odd}}
\newcommand{\ord}{\operatorname{ord}}

\newcommand{\cl}{\mathcal{C}}

\newcommand{\rb}{\mathbb{R}}

\newcommand{~}{\quad}
\newcommand{\cb}{\mathbb{C}}

\newcommand{\vep}{\varepsilon}

\definecolor{cardinalred}{RGB}{140,21,21}
\definecolor{coolgray}{RGB}{77,79,83}
\definecolor{black}{RGB}{0,0,0}
\definecolor{beige}{RGB}{210,194,149}
\definecolor{darkbeige}{RGB}{179,153,93}
\definecolor{darkcardinal}{RGB}{94,48,50}
\definecolor{lightcardinal}{RGB}{141,60,30}
\definecolor{darkpurple}{RGB}{83,40,79}
\definecolor{darkcyan}{RGB}{0,124,146}
\definecolor{skyblue}{RGB}{0,152,219}
\definecolor{seablue}{RGB}{10,100,180}
\definecolor{darkblue}{RGB}{20,80,150}
\definecolor{treegreen}{RGB}{0,155,118}
\definecolor{darkorange}{RGB}{168,101,12}
\definecolor{beigegray}{RGB}{95,87,79}
\definecolor{boxgray}{RGB}{238,235,233}
\definecolor{footergray}{RGB}{199,209,197}

\title[Real polynomial Hurwitz numbers]{Real polynomial Hurwitz numbers}

\author{Yanqiao Ding}

\address{School of Mathematics and Statistics, Zhengzhou University, Zhengzhou, 450001, China}

\email{yqding@zzu.edu.cn}



\date{\today}

\begin{document}

\maketitle

\begin{abstract}
    We show that the signed counts of normalized real polynomials, as defined by Itenberg and Zvonkine,
    provide the signed counts of genus zero real ramified coverings of the Riemann sphere with a point of total ramification
    and several other branch points with arbitrary ramification profiles.
\end{abstract}

\tableofcontents

\section{Introduction}
Real enumerative geometry has experienced rapid development since Welschinger's seminal work \cite{wel2005a,wel2005b}.
Where he defined integers that signed count real rational curves of fixed degree passing through prescribed points in ambient varieties of complex dimension $2$ or $3$.
These integers are independent of point positions and are known as \textit{Welschinger invariants}.
Mikhalkin made a significant discovery that the Gromov-Witten invariants and Welschinger invariants of toric surfaces can be computed by weighted counts of tropical curves with different multiplicities in \cite{mikhalkin05}.
Inspired by Welschinger invariants, the real Gromov-Witten invariants were defined \cite{gz2018},
which virtually count real curves of higher genus in some varieties of higher dimension.
The question considered in this note is whether there exists an analogue of Welschinger invariants
when the ambient variety is specifically an algebraic curve.

Hurwitz numbers are analogues of Gromov-Witten invariants when the ambient space is the projective line $\cb P^1$.
Hurwitz numbers enumerate the ramified coverings of $\cb P^1$ by genus $g$ Riemann surfaces with prescribed ramification profiles at certain branch points.
A \textit{real structure} on a Riemann surface $\Sigma$ is an anti-holomorphic involution $\tau$. 
A ramified covering $\pi:(\Sigma,\tau)\to(\cb P^1,\conj)$ is \textit{real} if it preserves the real structures, that is, $\pi\circ\tau=\conj\circ\pi$,
where $\conj$ is the complex conjugation on $\cb P^1$.
The real Hurwitz numbers are the counterpart of the Hurwitz numbers that enumerate real ramified coverings of $\cb P^1$.
These numbers usually depend on the positions of the real branch points.
Similar to the computation of Welschinger invariants, tropical geometry is used to compute real Hurwitz numbers \cite{bbm-2011,mr-2015,gpmr-2015},
and to construct lower bounds for real double Hurwitz numbers \cite{rau2019,d-2020,ding-2024,DLLY-2024}. 
Note that these lower bounds are based on the tropical computation in \cite{mr-2015}, 
not on an invariant signed count such as the Welschinger invariants.
Finding a Welschinger-type signed count definition for real Hurwitz numbers is highly worthwhile.

Itenberg and Zvonkine introduced the \textit{$s$-numbers}, 
which are signed counts of normalized real polynomials with given branch points and ramification profiles \cite{iz-2018}.
They showed that the $s$-numbers are independent of the positions of the branch points. 
The $s$-numbers were generalized to the case of real simple rational functions in \cite{er-2019}.
It is well known that counting normalized polynomials with given branch points and ramification profiles is
equivalent to computing the number of ramified coverings of the sphere by the sphere with a point of total ramification and several other branch points
with given ramification profiles.
In contrast to the complex case, the difference between counting real ramified coverings of the sphere
and normalized real polynomials is more subtle,
because a real polynomial can be normalized through a real change of variables in $0$, $1$ or $2$ ways,
depending on the parity of its degree and the sign of its leading coefficient \cite[Remark 1.8]{iz-2018}.
Hence, the signed counts of real normalized polynomials do not imply signed counts of real ramified coverings automatically.

In this note, we present an observation (see Theorem \ref{thm:main}):
the signed counts of normalized real polynomials defined by Itenberg and Zvonkine
actually provide signed counts of real ramified coverings of the sphere with a fully ramified branch point and several other branch points with arbitrary ramification profiles.
We call these signed counts of real ramified coverings the 
\textit{real polynomial Hurwitz numbers}.
The real polynomial Hurwitz numbers provide a signed count definition for genus zero one-part real double Hurwitz numbers. 
Their generating series exhibit polynomiality (see Corollary \ref{cor:gen-series}).
As the degree tends to infinity and only simple branch points are added, 
the signed counting defined genus zero one-part real double Hurwitz numbers are logarithmically equivalent
to the complex genus zero one-part double Hurwitz numbers (see Corollary \ref{cor:asymptotics}).

\section{Proof of the main results}

\subsection{Polynomial Hurwitz numbers and $s$-numbers}  
We recall the definitions of polynomial Hurwitz numbers and the $s$-numbers of normalized real polynomials. 
The readers may refer to \cite{cm-2016,iz-2018} for more details.

\begin{definition}
\label{def:pHur}
The \textit{polynomial Hurwitz numbers} are defined as
\begin{equation}\label{eq:pHur}
H(\lambda_1,\ldots,\lambda_k):=\sum_{[\pi]}\frac{1}{|\aut(\pi)|},
\end{equation}
where the sum runs over isomorphism classes of degree $d$ holomorphic ramified coverings
$\pi:S\to\cb P^1$ satisfying the following conditions.
\begin{enumerate}
\item $S$ is a connected Riemann sphere.
\item The map $\pi$ is fully ramified over $\infty$, and is ramified over $w_1,\ldots,w_k$ with ramification profiles $\lambda_1,\ldots,\lambda_k$, respectively.
Here, $w_1,\ldots,w_k$ are $k$ points in $\cb P^1$.
\item The map $\pi$ is not ramified elsewhere.
\item $\lambda_1,\ldots,\lambda_k$ are $k$ partitions of $d$ such that
$\sum\limits_{i=1}^kl(\lambda_i)=(k-1)d+1$.
\end{enumerate}
Here, two ramified coverings $\pi_1:S_1\to\cb P^1$ and $\pi_2:S_2\to\cb P^1$ are
isomorphic if there exists a biholomorphic map $\varphi:S_1\to S_2$ such that $\pi_1=\pi_2\circ\varphi$. Note that $\aut(\pi)$ denotes the automorphism group of the ramified covering map $\pi:S\to\cb P^1$.
\end{definition}
The number $H(\lambda_1,\ldots,\lambda_k)$ in equation (\ref{eq:pHur}) is not dependent on the positions of the branch points $w_1,\ldots,w_k$.
A polynomial $P\in\cb[x]$ is \textit{normalized}, if it has the form $P(z)=z^n+a_2z^{n-2}+\cdots+a_n$, where $a_i\in\cb$.
It is well known that the number $H(\lambda_1,\ldots,\lambda_k)$ equivalently counts the number of normalized polynomials with $k$ branch points such that the ramification profiles are $\lambda_1,\ldots,\lambda_k$, respectively.
When the branch points $w_1,\ldots,w_k$ are contained in $\rb P^1\subset\cb P^1$,
the real analogue of $H(\lambda_1,\ldots,\lambda_k)$ depends on the positions of $w_1,\ldots,w_k$.


\begin{definition}[{\cite[Definition 1.1-1.2]{iz-2018}}]
\label{def:disorder}
Let $P\in\rb[x]$ be a normalized real polynomial.
A pair of real numbers $x_1<x_2$ is called a \textit{disorder} of $P$ at $w$
if $P(x_1)=P(x_2)=w$ and the ramification order of $P$ at $x_1$ is greater than that at $x_2$.
Let $t(P)$ be the total number of disorders in $P$.
The \textit{sign} $\vep(P)$ of $P$ is defined to be $\vep(P)=(-1)^{t(P)}$.
\end{definition}

Let $k,d$ be two positive integers with $k<d$. Let $w_1<\ldots<w_k$ be a sequence of pairwise distinct real numbers.
Choose a sequence of partitions $\lambda_1,\ldots,\lambda_k$ of $d$ such that 
$$
\sum\limits_{i=1}^kl(\lambda_i)=(k-1)d+1.
$$
In the following of this note, we always assume that $k,d,w_1,\ldots,w_k$ and $\lambda_1,\ldots,\lambda_k$ satisfy the above assumption.

\begin{definition}[{\cite[Definition 1.3]{iz-2018}}]
\label{def:s-number}
Let $S_{\lambda_1,\ldots,\lambda_k}(w_1,\ldots,w_k)$ denote the set of real normalized polynomials $P$ that ramify at $w_1,\ldots,w_k$
with ramification profiles $\lambda_1,\ldots,\lambda_k$, respectively.
The \textit{$s$-number} of real polynomials with ramification profiles $\lambda_1,\ldots,\lambda_k$ is the sum
\begin{equation}\label{eq:s-number}
s(\lambda_1,\ldots,\lambda_k):=\sum_{P\in S_{\lambda_1,\ldots,\lambda_k}(w_1,\ldots,w_k)}\vep(P).
\end{equation}
\end{definition}
From \cite[Theorem 1]{iz-2018}, the $s$-number does not depend on the order of the real numbers $w_1,\ldots,w_k$ in the real line,
but only on the partitions $\lambda_1,\ldots,\lambda_k$.

\subsection{Real polynomial Hurwitz numbers}
In this section, we introduce the signed counts of genus zero real ramified coverings of the Riemann sphere with a fully ramified branch point
and several other branch points with prescribed ramification types.

Let $P,Q\in\rb[x]$ be two real polynomials of degree $d$.
If there exists a real affine transformation $\varphi:x\mapsto ax+b$, with $a\in\rb^*$ and $b\in\rb$,
such that $P=Q\circ\varphi$, the two real polynomials $P$ and $Q$ are said \textit{real isomorphic}.
Denote by $[P]_\rb$ the real isomorphism class of the real polynomial $P$.
Let $\aut^\rb(P)$ denote the real automorphism group of $P$. 

\begin{lemma}\label{lem:nom-rep}
Let $P(x)=a_dx^d+a_{d-1}x^{d-1}+\cdots+a_0$ be a real polynomial of degree $d$.
\begin{enumerate}
\item When $d$ is odd, there is a unique normalized real polynomial $Q$ in $[P]_\rb$ and $\aut^\rb(Q)$ is trivial.
\item When $d$ is even and $a_d>0$, there is $1$ or $2$ normalized real polynomials in $[P]_\rb$.
\begin{itemize}
\item If there is one normalized real polynomial $Q$ in $[P]_\rb$, we have $|\aut^\rb(Q)|=2$.
\item If there are two normalized real polynomials $Q_1,Q_2$ in $[P]_\rb$, we obtain that $\aut^\rb(Q_1)$ and $\aut^\rb(Q_2)$ are both trivial.
\end{itemize}
\item When $d$ is even and $a_d<0$, there is $1$ or $2$ normalized real polynomials in $[-P]_\rb$.
\begin{itemize}
\item If there is one normalized real polynomial $Q$ in $[-P]_\rb$, we have $|\aut^\rb(Q)|=2$.
\item If there are two normalized real polynomials $Q_1,Q_2$ in $[-P]_\rb$, we obtain that $\aut^\rb(Q_1)$ and $\aut^\rb(Q_2)$ are both trivial.
\end{itemize}
\end{enumerate}
\end{lemma}

\begin{proof}
When $d$ is odd, the polynomial $P$ can be normalized by a unique affine transformation
$z\mapsto\sign(a_d)\left(\frac{1}{|a_d|}\right)^{\frac{1}{d}}z-\frac{a_{d-1}}{da_d}$. Hence,
there is a unique normalized real polynomial $Q$ in $[P]_\rb$ and $\aut^\rb(Q)$ is trivial.

When $d$ is even and $a_d>0$, the polynomial $P$ can be normalized by two affine transformations
$z\mapsto\pm\left(\frac{1}{a_d}\right)^{\frac{1}{d}}z-\frac{a_{d-1}}{da_d}$.
Suppose that $P$ is normalized to $Q_1$ and $Q_2$. If $Q_1\neq Q_2$,
there are two normalized real polynomials in $[P]_\rb$ and $|\aut^\rb(Q_1)|=|\aut^\rb(Q_1)|=1$.
If $Q_1=Q_2=Q$, the coefficients of the odd degree monomials in $Q$ are all zero.
There is only one normalized real polynomial in the set $[P]_\rb$ and $|\aut^\rb(Q)|=2$.
In the case that $d$ is even and $a_d<0$, the above argument also implies the statement $(3)$.
\end{proof}

A holomorphic function $P:\cb P^1_{\sou}\to\cb P^1_{\targ}$ is \textit{polynomial}, if $P^{-1}(\infty_{\targ})=\infty_{\sou}$.
Let $S$ be a connected Riemann sphere.
The uniqueness of the complex structure on a $2$-dimensional sphere implies that
there exists a biholomorphic isomorphism $\varphi:S\to\cb P^1$.
The standard complex conjugation $\conj$ on $\cb P^1$ induces a real structure on $S$ via $\varphi$.
We use $\conj_S$ to denote this complex conjugation on $S$.

Let $\cl^\rb(\lambda_1,\ldots,\lambda_k)$ denote the set of degree $d$ holomorphic ramified coverings $\pi:S\to\cb P^1$ 
that satisfy conditions $(1)-(4)$ in Definition \ref{def:pHur} and preserve the real structures, \textit{i.e.} $\pi\circ\conj_S=\conj\circ\pi$.
Denote by $\cl^\rb(\lambda_1,\ldots,\lambda_k)/\sim$ the set of real isomorphism classes of real ramified coverings 
$(\pi,\conj_S)$ in $\cl^\rb(\lambda_1,\ldots,\lambda_k)$.
Let $S_{\lambda_1,\ldots,\lambda_k}(w_1,\ldots,w_k)/\sim$ be the set of real affine isomorphism classes of 
real normalized polynomials in $S_{\lambda_1,\ldots,\lambda_k}(w_1,\ldots,w_k)$.

\begin{lemma}\label{lem:corr}
If $d>1$ is an odd integer, we have a bijection 
$$
\Phi_{\lambda_1,\ldots,\lambda_k}:\cl^\rb(\lambda_1,\ldots,\lambda_k)/\sim\to S_{\lambda_1,\ldots,\lambda_k}(w_1,\ldots,w_k).
$$
If $d>0$ is an even integer, there is a bijection
$$
\Phi_{\lambda_1,\ldots,\lambda_k}:\cl^\rb(\lambda_1,\ldots,\lambda_k)/\sim\to S_{\lambda_1,\ldots,\lambda_k}(w_1,\ldots,w_k)/\sim\sqcup S_{\lambda_k,\ldots,\lambda_1}(-w_k,\ldots,-w_1).
$$
\end{lemma}

\begin{proof}
Let $\pi:(S,\conj_S)\to(\cb P^1,\conj)$ be a degree $d$ real ramified covering in $\cl^\rb(\lambda_1,\ldots,\lambda_k)$.
The real ramified covering $\pi:(S,\conj_S)\to(\cb P^1,\conj)$ is real isomorphic to a real polynomial function
$P:(\cb P^1_{\sou},\conj)\to(\cb P^1_{\targ},\conj)$.
Moreover, two real isomorphic ramified coverings $\pi_1,\pi_2\in\cl^\rb(\lambda_1,\ldots,\lambda_k)$ are real isomorphic to two real polynomial functions $P_1,P_2$ such that $P_1\circ\varphi=P_2$, where $\varphi:(\cb P^1_{\sou},\conj)\to(\cb P^1_{\sou},\conj)$ is a real biholomorphism with $\varphi(\infty_{\sou})=\infty_{\sou}$.
We fix an identification $\cb P^1=\cb\sqcup\{\infty\}$.
The real polynomial functions $P$ are entirely determined by their restriction to the affine plane $P:\cb_{\sou}\to\cb_{\targ}$.
Let $P_{\lambda_1,\ldots,\lambda_k}(w_1,\ldots,w_k)$ denote the set of real polynomials that ramified at $w_1,\ldots,w_k$ with ramification profiles $\lambda_1,\ldots,\lambda_k$, respectively.
The above argument implies that there exists a bijection 
$$
\Psi_{\lambda_1,\ldots,\lambda_k}:\cl^\rb(\lambda_1,\ldots,\lambda_k)/\sim\to P_{\lambda_1,\ldots,\lambda_k}(w_1,\ldots,w_k)/\sim.
$$

From Lemma \ref{lem:nom-rep}, when $d$ is odd, the set $P_{\lambda_1,\ldots,\lambda_k}(w_1,\ldots,w_k)/\sim$
is bijective to the set $S_{\lambda_1,\ldots,\lambda_k}(w_1,\ldots,w_k)$.
We use $\Psi_{\lambda_1,\ldots,\lambda_k}'$ to denote this bijection.
Then $\Phi_{\lambda_1,\ldots,\lambda_k}=\Psi_{\lambda_1,\ldots,\lambda_k}'\circ\Psi_{\lambda_1,\ldots,\lambda_k}$ is the required bijection.
If $d$ is even, the set $P_{\lambda_1,\ldots,\lambda_k}(w_1,\ldots,w_k)/\sim$ consists of two disjoint subsets 
$$
P^+_{\lambda_1,\ldots,\lambda_k}(w_1,\ldots,w_k)/\sim\sqcup P^-_{\lambda_1,\ldots,\lambda_k}(w_1,\ldots,w_k)/\sim,
$$
where $P^{\pm}_{\lambda_1,\ldots,\lambda_k}(w_1,\ldots,w_k)/\sim$ is the subset consisting of real polynomials
whose leading coefficients have the sign $\pm$.
From Lemma \ref{lem:nom-rep}, $P^+_{\lambda_1,\ldots,\lambda_k}(w_1,\ldots,w_k)/\sim$ is bijective to
$S_{\lambda_1,\ldots,\lambda_k}(w_1,\ldots,w_k)/\sim$, and $P^-_{\lambda_1,\ldots,\lambda_k}(w_1,\ldots,w_k)/\sim$
is bijective to $S_{\lambda_k,\ldots,\lambda_1}(-w_k,\ldots,-w_1)$.
We use $\Psi_{\lambda_1,\ldots,\lambda_k}'$ to denote this bijection,
then $\Phi_{\lambda_1,\ldots,\lambda_k}=\Psi_{\lambda_1,\ldots,\lambda_k}'\circ\Psi_{\lambda_1,\ldots,\lambda_k}$ is the required bijection.
\end{proof}

Let $\lambda$ be a partition of an even integer $d$.
We use $o(\lambda)$ to denote the number of integers that appear an odd number of times in $\lambda$.

\begin{lemma}\label{lem:sign}
    Let $d>0$ be an even integer, and let $\lambda_1,\ldots,\lambda_k$ be $k$ partitions of $d$ such that $\sum\limits_{i=1}^kl(\lambda_i)=(k-1)d+1$.
    Suppose that $(\pi,\conj_S)\in\cl^\rb(\lambda_1,\ldots,\lambda_k)$ is a real ramified covering such that $|\Phi_{\lambda_1,\ldots,\lambda_k}([\pi,\conj_S])|=2$.
    Then, we have the following statements.
    \begin{itemize}
        \item If $\sum\limits_{i=1}^k\lfloor\frac{o(\lambda_i)}{2}\rfloor$ is an even number, the sign $\vep(P_1)=\vep(P_2)$, where $P_1$, $P_2$ are the two polynomials in $\Phi_{\lambda_1,\ldots,\lambda_k}([\pi,\conj_S])$.
        \item If $\sum\limits_{i=1}^k\lfloor\frac{o(\lambda_i)}{2}\rfloor$ is an odd number, the sign $\vep(P_1)=-\vep(P_2)$.
    \end{itemize}
\end{lemma}

\begin{proof}
    Let $P_1$ and $P_2$ be the two real isomorphic normalized polynomials of the class $\Phi_{\lambda_1,\ldots,\lambda_k}([\pi,\conj_S])$.
    Then $P_1=P_2\circ\varphi$, where $\varphi:z\mapsto-z$.
    A pair of real numbers $x_1<x_2$ is called an \textit{ordered pair} of $P_2$ at $w$ if $P_2(x_1)=P_2(x_2)=w$ and the ramification order of $P_2$ at $x_2$ is greater than that at $x_1$.
    Let $\ord(P_2)$ be the total number of ordered pairs of $P_2$.
    Since $P_1=P_2\circ\varphi$, we obtain $t(P_1)=\ord(P_2)$.
    Let $S_{w_i}$, $i=1,\ldots,k$, be the number of disorders and ordered pairs of $P_2$ at the critical value $w_i$.
    In the following, we prove that if $\sum\limits_{i=1}^k\lfloor\frac{o(\lambda_i)}{2}\rfloor$ is an even number,
    then the sum $t(P_2)+\ord(P_2)$ is even. Otherwise, the sum $t(P_2)+\ord(P_2)$ is odd.

    Suppose that the numbers in the set $A=\{a_1,\dots,a_l\}= P_2^{-1}(w_1)\cap\rb$ satisfy $a_1<\cdots<a_l$.
    Denote by $r_i$ the ramification order of $P_2$ at $a_i$, where $i=1,\ldots,l$.
    Suppose that the sequence $R=(r_1,\ldots,r_l)$ contains $n$ distinct integers $t_1,\ldots,t_n$.
    We assume that $t_1,\ldots,t_{o(\lambda_1)}$ are the $o(\lambda_1)$ integers that appear an odd number of times in $\lambda_1$.
    The number of disorders (resp. ordered pairs) of $P_2$ at the critical value $w_1$ is the number of disorders (resp. ordered pairs) in the sequence $R$.
    The number of disorders and ordered pairs in the sequence $R$ 
    that involve the number $t_1$ is $n_{t_1}(R)\cdot(l-n_{t_1}(R))$,
    where $n_{t_1}(R)$ is the number of times $t_1$ appears in the sequence $R$.
    The number of disorders and ordered pairs in the sequence $R$ 
    that involve the number $t_2$ and do not involve $t_1$ is $n_{t_2}(R)\cdot(l-n_{t_1}(R)-n_{t_2}(R))$.
    Repeating this procedure, we find that the number $S_{w_1}$ of disorders and ordered pairs of $P_2$ at the critical value $w_1$ is given by:
    $$
    S_{w_1}=\sum_{i=1}^n\left(n_{t_i}(R)\cdot(l-\sum_{j=1}^{i}n_{t_j}(R))\right).
    $$

    Suppose that $o(\lambda_1)$ is odd. 
    Since the number of elements in the set $P_2^{-1}(w_1)\setminus A$ is even, the number $l$ is odd.
    Note that $n_{r_i}(R)$ is odd for $i\in\{1,\ldots,o(\lambda_1)\}$
    and is even for $i\in\{o(\lambda_1)+1,\ldots,l\}$.
    The number $l-\sum\limits_{j=1}^in_{r_j}(R)$ is even if $i$ is an odd number in the set $\{1,\ldots,o(\lambda_1)\}$ and is odd if $i$ is even in $\{1,\ldots,o(\lambda_1)\}$.
    Therefore, the sum $S_{w_1}$ is even (resp. odd) if the number of even integers in the set $\{1,\ldots,o(\lambda_1)\}$ is even (resp. odd), that is, $\lfloor\frac{o(\lambda_1)}{2}\rfloor$ is even (resp. odd).
    When $o(\lambda_1)$ is even, we obtain that $S_{w_1}$ is even (resp. odd) if $\frac{o(\lambda_1)}{2}$ is even (resp. odd) via an argument similar to the above.
    In summary, $S_{w_1}$ and $\lfloor\frac{o(\lambda_1)}{2}\rfloor$ have the same parity. Of course, the same conclusion holds for $S_{w_i}$ and $\lfloor\frac{o(\lambda_i)}{2}\rfloor$, where $i=2,\ldots,k$.
    Note that
    $$
    t(P_2)+\ord(P_2)=S_{w_1}+\cdots+S_{w_k}.
    $$
    Therefore, $t(P_2)+\ord(P_2)$ has the same parity as $\sum\limits_{i=1}^k\lfloor\frac{o(\lambda_i)}{2}\rfloor$.
\end{proof}

\begin{definition}\label{def:real-sign}
Let $[\pi,\conj_S]$ be a real isomorphism class of a degree $d$ ramified covering in $\cl^\rb(\lambda_1,\ldots,\lambda_k)$.
\begin{itemize}
    \item If $d$ is odd, 
    the \textit{sign} $\vep([\pi,\conj_S])$ of $[\pi,\conj_S]$ is defined as $\vep(P)$, 
    where $P$ is the normalized real polynomial in the class $\Phi_{\lambda_1,\ldots,\lambda_k}([\pi,\conj_S])$.
    \item If $d$ and $\sum\limits_{i=1}^k\lfloor\frac{o(\lambda_i)}{2}\rfloor$ are all even, 
    the \textit{sign} $\vep([\pi,\conj_S])$ of $[\pi,\conj_S]$ is defined as $\vep(P)$, where $P$ is a real normalized polynomial in the class $\Phi_{\lambda_1,\ldots,\lambda_k}([\pi,\conj_S])$.
\end{itemize}
\end{definition}

From Lemma \ref{lem:corr} and Lemma \ref{lem:sign}, the sign of a real isomorphism class $[\pi,\conj_S]$ is well defined.

\begin{definition}\label{def:real-pHur}
Let $k,d$ be two positive integers with $k<d$.
Choose a sequence of partitions $\lambda_1,\ldots,\lambda_k$ of $d$ such that
$\sum\limits_{i=1}^kl(\lambda_i)=(k-1)d+1$.
\begin{itemize}
    \item If $d$ is odd or $d$ and $\sum\limits_{i=1}^k\lfloor\frac{o(\lambda_i)}{2}\rfloor$ are even,
    the \textit{real polynomial Hurwitz numbers} are defined as
    \begin{equation}\label{eq:real-pHur}    H^\rb(\lambda_1,\ldots,\lambda_k):=\sum_{[\pi,\conj_S]}\frac{\vep([\pi,\conj_S])}{|\aut^\rb(\pi,\conj_S)|}.
    \end{equation}
    Here, the sum is taken over real isomorphism classes in $\cl^\rb(\lambda_1,\ldots,\lambda_k)/\sim$,
    and $\aut^\rb(\pi,\conj_S)$ is the group of real automorphisms of $(\pi,\conj_S)$.
    \item If $d$ is even and $\sum\limits_{i=1}^k\lfloor\frac{o(\lambda_i)}{2}\rfloor$ is odd,
    the real polynomial Hurwitz numbers are defined as
    $$
    H^\rb(\lambda_1,\ldots,\lambda_k)=0.
    $$
\end{itemize}
\end{definition}

\begin{theorem}\label{thm:main}
Let $k,d$ be two positive integers with $k<d$. Choose a sequence of pairwise distinct real numbers $w_1,\ldots,w_k$ that satisfy $w_1<\ldots<w_k$.
Let $\lambda_1,\ldots,\lambda_k$ be a sequence of partitions of $d$ such that
$\sum\limits_{i=1}^kl(\lambda_i)=(k-1)d+1$.
Then, we have
$$
H^\rb(\lambda_1,\ldots,\lambda_k)=s(\lambda_1,\ldots,\lambda_k).
$$
\end{theorem}

\begin{proof}
If $d$ is odd,from Lemma \ref{lem:corr} and Lemma \ref{lem:nom-rep}(1), 
$|\Phi([\pi,\conj_S])|=1$, $|\aut^\rb(\pi,\conj_S)|=1$ and 
$\Phi_{\lambda_1,\ldots,\lambda_k}(\cl^\rb(\lambda_1,\ldots,\lambda_k)/\sim)=S_{\lambda_1,\ldots,\lambda_k}(w_1,\ldots,w_k)$.
Hence, we have 
$$
H^\rb(\lambda_1,\ldots,\lambda_k)=s(\lambda_1,\ldots,\lambda_k).
$$

If $d$ is even and $\sum\limits_{i=1}^k\lfloor\frac{o(\lambda_i)}{2}\rfloor$ is odd, 
there exists at least one $\lambda_i$ in the $k$ partitions $\lambda_1,\ldots,\lambda_k$ such that $o(\lambda_i)\geq2$.
Suppose that $t_1,\ldots,t_{o(\lambda_i)}$ are the $o(\lambda_i)$ numbers that appear an odd number of times in $\lambda_i$.
Since $d$ is even, the number of odd integers, denoted by $c$,
in the set $\{t_1,\ldots,t_{o(\lambda_i)}\}$ is even.
If $c=0$, there are at least two even integers that appear an odd number of times in $\lambda_i$.
Let $\lambda_i'$ be the partition obtained from $\lambda_i$ by subtracting 1 from every element of $\lambda_i$ and eliminating the zeros.
There are at least two odd numbers that appear an odd number of times in $\lambda_i'$.
From \cite[Theorem 3]{iz-2018}, $s(\lambda_1,\ldots,\lambda_k)=0$.
If $c$ is a positive even number, there is at least one even number that appears an odd number of times in $\lambda_i'$.
We have $s(\lambda_1,\ldots,\lambda_k)=0$ from \cite[Theorem 3]{iz-2018}.
Therefore, we obtain
$$
H^\rb(\lambda_1,\ldots,\lambda_k)=s(\lambda_1,\ldots,\lambda_k)=0.
$$

If $d$ and $\sum\limits_{i=1}^k\lfloor\frac{o(\lambda_i)}{2}\rfloor$ are both even, Lemma \ref{lem:corr} and Lemma \ref{lem:nom-rep} imply that:
\begin{itemize}
    \item when $|\Phi([\pi,\conj_S])|=2$ the number $|\aut^\rb(\pi,\conj_S)|$ is equal to 1;
    \item when $|\Phi([\pi,\conj_S])|=1$ the number $|\aut^\rb(\pi,\conj_S)|$ is equal to 2.
\end{itemize}
In the case $|\Phi([\pi,\conj_S])|=2$, we suppose that $P_1$, $P_2$ are the two real normalized polynomials.
From Lemma \ref{lem:sign}, the sign $\vep(P_1)=\vep(P_2)$,
so $\vep([\pi,\conj_S])=\frac{1}{2}(\vep(P_1)+\vep(P_2))$.
Hence, for any real isomorphism class $[\pi,\conj_S]$ we get
\begin{equation}\label{eq:av-sign}
    \vep([\pi,\conj_S])=\frac{1}{|\Phi_{\lambda_1,\ldots,\lambda_k}([\pi,\conj_S])|}\sum_{P\in\Phi_{\lambda_1,\ldots,\lambda_k}([\pi,\conj_S])}\vep(P).
\end{equation}
From the definition of real polynomial Hurwitz number we have 
\begin{align*}
H^\rb(\lambda_1,\ldots,\lambda_k)&=\frac{1}{2}\sum_{[\pi,\conj_S]}\sum_{P\in\Phi([\pi,\conj_S])}\vep(P)\\
&=\frac{1}{2}\left(\sum_{P\in S_{\lambda_1,\ldots,\lambda_k}(w_1,\ldots,w_k)}\vep(P)+\sum_{P\in S_{\lambda_k,\ldots,\lambda_1}(-w_k,\ldots,-w_1)}\vep(P)\right)\\
&=\frac{1}{2}\left(s(\lambda_1,\ldots,\lambda_k)+s(\lambda_k,\ldots,\lambda_1)\right).
\end{align*}
From \cite[Theorem 1]{iz-2018}, we get $H^\rb(\lambda_1,\ldots,\lambda_k)=s(\lambda_1,\ldots,\lambda_k)$.
\end{proof}

From Theorem \ref{thm:main} and \cite[Theorem 1]{iz-2018}, one finds that the signed counts
$H^\rb(\lambda_1,\ldots,\lambda_k)$ do not depend on the real numbers $w_1,\ldots,w_k$
(in particular, do not depend on the order of the numbers $w_1,\ldots,w_k$),
but only on the partitions $\lambda_1,\ldots,\lambda_k$.

\begin{remark}
It is interesting to consider the signed counts of the real ramified coverings corresponding
to the signed counts of the real simple rational functions, as defined by El Hilany and Rau in \cite{er-2019}. 
\end{remark}

\begin{remark}
When $d$ is even, from the proof of Theorem \ref{thm:main}, the sign of a real isomorphism class can alternatively be defined by equation (\ref{eq:av-sign}).
If $\sum\limits_{i=1}^k\lfloor\frac{o(\lambda_i)}{2}\rfloor$ is odd and $|\Phi_{\lambda_1,\ldots,\lambda_k}([\pi,\conj_S])|=2$,
equation (\ref{eq:av-sign}) implies the sign $\vep([\pi,\conj_S])=0$.
The proof of Theorem \ref{thm:main} implies that if we use equation (\ref{eq:av-sign}) as the definition of the sign of a real isomorphism class, the number $H^\rb(\lambda_1,\ldots,\lambda_k)$ is also equal to $s(\lambda_1,\ldots,\lambda_k)$.
This definition has no restriction on the parity of the sum $\sum\limits_{i=1}^k\lfloor\frac{o(\lambda_i)}{2}\rfloor$, but allows some real isomorphism classes to be counted with zero weight.
From \cite[Theorem 3]{iz-2018}, if $s(\lambda_1,\ldots,\lambda_k)\neq0$,
the partitions $\lambda_1,\ldots,\lambda_k$ have to satisfy the condition that $\sum\limits_{i=1}^k\lfloor\frac{o(\lambda_i)}{2}\rfloor$ is even.
Hence, we define the sign of a real isomorphism class as Definition \ref{def:real-sign}.
\end{remark}

\subsection{One-part real double Hurwitz numbers}
From Theorem \ref{thm:main} and \cite[Theorem 2-5]{iz-2018}, it is obvious that $H^\rb(\lambda_1,\ldots,\lambda_k)$ satisfies the same properties as the $s$-numbers.
In this section, we restrict our attention to genus zero one-part real double Hurwitz numbers,
where $\lambda_1$ is a partition of $d$ and $\lambda_2,\ldots,\lambda_k$ are equal to $(2,1,\ldots,1)$.

The signed counting defined genus zero one-part real double Hurwitz number is defined as
$$
H^\rb_0(\lambda):=H^\rb(\lambda,(2,1,\ldots,1),\ldots,(2,1,\ldots,1)).
$$
The following generating series is defined in \cite[Section 1.3]{iz-2018}.
The generating series for genus zero one-part real double Hurwitz numbers is given by:
$$
F_{\lambda}(q):=\sum_{m\geq 0}h^\rb_\lambda(m)\frac{q^m}{m!},
$$
where $h^\rb_\lambda(m)=H^\rb_0(\lambda\cup(1^m))$.
The notation $\lambda\cup(1^m)$ denotes the partition obtained by adding $m$ ones to the partition $\lambda$.
The generating series $F_\lambda(q)$ is decomposed into an even part and an odd part.
The even part $F^{\even}_\lambda(q)$ (resp. odd part $F^{\odd}_\lambda(q)$) enumerates real ramified coverings of even (resp. odd) degrees.
More precisely, we have
$$
F^{\even}_\lambda(q)=\sum_{\substack{m\geq0 \\m+|\lambda| \text{ is even}}}h^\rb_\lambda(m)\frac{q^m}{m!},~
F^{\odd}_\lambda(q)=\sum_{\substack{m\geq0 \\m+|\lambda| \text{ is odd}}}h^\rb_\lambda(m)\frac{q^m}{m!}.
$$

\begin{corollary}\label{cor:gen-series}
For any partition $\lambda$, we have the following.
\begin{enumerate}
\item The generating series $F_\lambda^{\even}(q)$ is a polynomial in $q$ and $f(q)$ with rational coefficients,
where $f(q)=\tanh(q)=\frac{e^q-e^{-q}}{e^q+e^{-q}}$.
\item The generating series $F_\lambda^{\odd}(q)$ is equal to $g(q)$ multiplied by a polynomial in $q$ and $f(q)$ with rational coefficients,
where $g(q)=\frac{1}{\cosh(q)}=\frac{2}{e^q+e^{-q}}$.
\end{enumerate}
\end{corollary}

\begin{proof}
It is straightforward to conclude the above statements from Theorem \ref{thm:main} and \cite[Theorem 2]{iz-2018}.
\end{proof}

\begin{remark}
    The polynomiality of the generating series $F_\lambda^{\even}(q)$ and $F_\lambda^{\odd}(q)$ is
    different from the polynomiality of the complex genus zero one-part double Hurwitz numbers in \cite{gjv-2005}.
\end{remark}

\begin{corollary}\label{cor:asymptotics}
Let $\lambda$ be a partition such that every odd number greater than $1$ appears an even number of times
and at most one even number appears an odd number of times.
Then, we have
\begin{align*}
\ln|h_\lambda^\rb(m)|\underset{\substack{m\to\infty \\m\text{ even}}}{\sim}&m\ln m~\text{for }|\lambda|\text{ even},\\
\ln|h_\lambda^\rb(m)|\underset{\substack{m\to\infty \\m\text{ odd}}}{\sim}&m\ln m~\text{for }|\lambda|\text{ odd}.
\end{align*}
Let $\lambda$ be a partition such that at most one even number appears an odd number of times
and at most one odd number greater than $1$ appears an odd number of times.
Then, we have
\begin{align*}
\ln|h_\lambda^\rb(m)|\underset{\substack{m\to\infty \\m\text{ even}}}{\sim}&m\ln m~\text{for }|\lambda|\text{ odd},\\
\ln|h_\lambda^\rb(m)|\underset{\substack{m\to\infty \\m\text{ odd}}}{\sim}&m\ln m~\text{for }|\lambda|\text{ even}.
\end{align*}
\end{corollary}

\begin{proof}
It is straightforward to obtain the above logarithmic asymptotics from Theorem \ref{thm:main} and \cite[Theorem 5]{iz-2018}.
\end{proof}
From \cite[Theorem 5.10]{rau2019}, the number $h_\lambda^\rb(m)$ is logarithmically equivalent to its complex counterpart.

\begin{remark}
Definition \ref{def:real-pHur} offers a signed counting method to define genus zero one-part real double Hurwitz numbers.
In a subsequent paper, we will explore extending this method to define ordinary genus zero real double Hurwitz numbers.
\end{remark}

\section*{Acknowledgements}
The author thanks Cong Mo and Xinyu Liu for their participation in the early stage of this project.
This work was partially supported by the National Natural Science Foundation of China (No.12101565).


\end{document}